\documentclass[12pt, a4paper]{amsart}
\usepackage{amsmath,amssymb,amscd,amsfonts}

\usepackage{ifthen}
\usepackage[T1]{fontenc}
\usepackage[utf8]{inputenc}
\usepackage[all]{xy}
\usepackage{graphicx}
\usepackage{enumerate}
\usepackage{xspace}
\usepackage{pdfsync}
\usepackage{epic}
\usepackage{dsfont}

\newtheorem{theorem}{Theorem}
\newtheorem{remark}[theorem]{Remark}
\newtheorem{definition}[theorem]{Definition}
\newtheorem{lemma}[theorem]{Lemma}

\newtheorem{proposition}[theorem]{Proposition}
\newtheorem{conjecture}[theorem]{Conjecture}
\newtheorem{corollary}[theorem]{Corollary}
\newtheorem{example}[theorem]{Example}
\newtheorem{question}[theorem]{Question}
\begin{document}
\title[]{Bogomolov decomposition and compact K\"ahler manifolds of algebraic dimension zero.}



\author{Fr\'ed\'eric Campana}
\address{Universit\'e Lorraine \\
 Institut Elie Cartan\\
Nancy, France}

\email{frederic.campana@univ-lorraine.fr}

\date{\today}
\maketitle

D\'edi\'e au Professeur F.A. Bogomolov, pour ses travaux fondamentaux.

\tableofcontents

\section{Some notation and terminology.}

$X$ is an $n$-dimensional compact connected complex manifold (sometimes just normal if explicitly said). A subvariety is an irreducible compact analytic subset $Z\subset X$. A subvariety is `trivial' if either a point or $X$, nontrivial otherwise. The `general' (resp. `generic') point of $X$ is a point lying outside countably (resp. finitely) many strict subvarieties of $X$ depending on the situation considered. 

A fibration $f:X\to Y$ is a surjective holomorphic map (regular model of a meromorphic fibration) with connected fibres. Its general fibre $X_y$ is the one lying over a general point $y$ of $Y$. Since we are considering bimeromorphic properties, the maps are meromorphic ones, their uniqueness is meant up to bimeromorphic equivalence. 

If bimeromorphic invariants such as: $q$ (the usual irregularity), $k$ (the Kummer dimension. See \S\ref{s k}), $a$ (the algebraic dimension, see next section), $s$ (the semi-simple dimension, see \S\ref{s ss}) of complex spaces in $C$ are defined, we denote with $q(f):=q(X/Y),k(f):=k(X/Y), a(f):=a(X/Y), s(f):=s(X/Y)$ the (well-defined) value of the invariant on the general fibre $X_y$ of a fibration $f:X\to Y$. This is, in general, {\bf not} the difference of the values of the invariant on $X$ and $Y$.

We denote by $BC(X)$ the Barlet-Chow cycle space of $X$, by $BC(X/Y)$ its Zariski-closed subset consisting of cycles contained in some fibre of $f:X\to Y$, equipped with its fibre map: $f_*:BC(X/Y)\to Y$. We refer to \cite{BM} for the construction and properties with proofs, of this space. The compactness of its components when $X$ is K\"ahler (resp. in the class $C$) is central in the present text.





\section{Introduction: Bimeromorphic classification of compact K\"ahler manifolds.}

Let $X_n$ be a compact complex manifold of dimension $n$. Let $0\leq a(X)\leq n$ be its `algebraic dimension', that is the transcendence degree over $\Bbb C$ of its field of meromorphic functions $M(X)$. If $a(X)=n$, $X$ is `Moishezon', that is: bimeromorphic to a projective manifold. At the other extreme, $a(X)=0$, and the number of irreducible divisors of $X$ is finite (\cite{K}). In the intermediate case, there is a unique (up to bimeromorphic equivalence) fibration $a_X:X\to A_X$, called the `algebraic reduction' of $X$, which induces a field isomorphism $a_X:M(A_X)\to M(X)$, and so $dim(A_X)=a(X)$. 

In the non-K\"ahler case, from dimension $3$ on, no classification scheme emerges, as shown by the extreme flexibility of twistor spaces of anti-self-dual riemannian fourfolds. One main reason being that covering analytic families of cycles of codimension $2$ or more do not produce meaningful fibrations, these families building non-compact components of the Barlet-Chow space $BC(X)$.

When $X$ is K\"ahler, or in the class C (i.e: bimeromorphic to K\"ahler\footnote{This class being better adapted to bimeromorphic geometry. It appeared for the first time in \cite{U}, Corollary 9.3, as the class of manifolds dominated by a compact K\"ahler one. The domination was proved to be chosen bimeromorphic in \cite{Va}.}), the situation is opposite: the space 
$BC(X)$ has compact components\footnote{As first observed and used by D. Lieberman (\cite{L}, theorem 1.1), with details and an extension to the Douady space by A. Fujiki in \cite{F78}.}, and covering families of cycles produce  fibrations. In particular,  algebraic reductions (\cite{Ca81}, \cite{F81}), as well as Albanese maps (\cite{Ca85}) can be constructed in the relative setting of fibrations.

$\bullet$ These two tools permit to essentially reduce the structure of any $X$ in C to the special case where $a(X)=0$, see \S\ref{s iM} for some details.

$\bullet$ The remaining task is then to describe the compact K\"ahler manifolds of algebraic dimension zero. We will give a quite short, but conditional answer (in Theorem \ref{t' a=0} below) to this question, in terms of `Kummer' and `simple non-Kummer' manifolds (introduced in \cite{F82}, see Definition \ref{d sk} below), the latter ones being conjecturally `bimeromorphically symplectic'. 

In short: assuming conjecture \ref{c simple} below, all compact K\"ahler manifolds are constructed, by means of a sequence of $4$ canonical fibrations from familiar ones: projective, tori and hyperk\"ahler. See Theorem \ref{t' a=0}, proved in \S\ref{s a=0}, for the case $a(X)=0$, and \S\ref{s gc} for the general case. This conjecture and decomposition were conceived in the mid 1980's, inspired by \cite{Bo}, \cite{Bo'}, \cite{Bo"}, and by \cite{F83} where a partial picture based on the Kummer reduction was presented. The semi-simple reduction was introduced in \cite{F82}, but not used later. We explain here how to-conditionally-obtain a complete picture of the bimeromorphic category of compact K\"ahler manifolds. Notice however that this description is for the individual $X$'s, and is extremely far from being preserved by deformations of the complex structure.

\begin{definition} \label{d sk}

\

1. $X_n$ is `simple' (\cite{F82}) if $n\geq 2$ and if $X$ is covered only by its trivial subvarieties. Alternatively, it has at most countably (finitely? see Question \ref{q max}) many maximal nontrivial subvarieties. If $X$ is simple, $a(X)=0$.

 $X_n, n\geq 2$ is `strictly simple' if its only subvarieties are trivial, i.e: either $X$ or its points.

2. $X$ and $X'\in C$ are `isogeneous' if $a(X)=a(X')=0$ and if there exists $X"\subset X\times X'$ irreducible generically finite onto $X$ and $X'$. Isogeny is an equivalence relation. The isogeny class of $X$ consists of all $X'\in C$ isogeneous to $X$.

3. $X$ is `semi-simple' if isogeneous to a product of simple manifolds, or, equivalently if $X=S/G$, where $S$ is a product of simple manifolds and $G$ a finite group of automorphisms of $S$.

$X$ is `isotypically semi-simple' if isogeneous to a product $S^k$ with $S$ simple. The isogeny types of $S$ and $X$ are then uniquely defined.

5. For tori, there are also classical definitions, which we call t-isogeny and t-simplicity, for which $T,T'$ are t-isogeneous (resp. $T$ is t-simple) if there is a common torus \'etale cover of both $T,T'$ (resp. if there is no nontrivial subtorus of $T$). 

It is easy to see that for tori $T,T'$ with $a(T)=a(T')=0$ isogeny and t-isogeny coincide, and that a torus $T$ is simple if and only if strictly simple, if and only if t-simple and not projective.

6. $X$ is Kummer if bimeromorphic to the quotient $T/G$ of a complex compact torus $T$ by the action of a finite group $G$ (\cite{U}, 16.1). If $a(X)=0$, this is equivalent to being isogeneous to a product of tori $T_j$ with $a(T_j)=0$.  A torus $T$ with $a(T)=0$ is in general not semi-simple, Poincar\'e reducibility being specific to Abelian varieties.

A semi-simple manifold $X$ is said to be `without Kummer factor' if no member of its isogeny class has a  Kummer (or Torus) quotient. 

7. $X$ is `bimeromorphically symplectic' if $n=2m$ is even, and if it has a holomorphic $2$-form $\sigma$ such that $\sigma^{\wedge m}\neq 0$. It is `irreducible' if, moreover, $h^{2,0}(X)=1$.

It were interesting to know whether or not an irreducible bimeromorphically symplectic manifold can be Kummer. See question \ref{q k}.

\end{definition}

\begin{example}

\

 1. The `simple' surfaces are bimeromorphic to either $K3$ or tori.

2. The strictly simple threefolds are \'etale quotients of simple tori (\cite{CDV}). The simple threefolds are Kummer (\cite{CHP}, \cite{GP} Appendix). 

3. A simple torus is strictly simple. The general torus of dimension $n\geq 2$ is simple (\cite{W}, Chap. VI, or \cite{BL},Chap. 1.7,1.10 ), and so are its Kummer quotients.

4. An irreducible  symplectic manifold $X$ with $a(X)=0$ is isotypically semi-simple (\cite{COP}, corollary 2.5). If $X$ is, moreover, simply connected without effective divisor, $X$ is simple (\cite{COP}, Corollary 2.6).  In particular, the general member of the Kuranishi family of Hyperk\"ahler manifolds is simple. The special members $K3^{[k]}$ and $K^{[k]}$ of \cite{B} are isotypically semi-simple.

5. The general member of the Kuranishi family of the Hilbert symmetric products $K3^{[k]}$ of a $K3$ surface is strictly simple (\cite{V}). 
\end{example}

\begin{conjecture}\label{c simple} 1. Let $X_n$ be a `simple' compact K\"ahler manifold . Then $X$ is either bimeromorphically irreducible symplectic, or Kummer. In particular, $X_n$ is Kummer if $n$ is odd.

2. Let $f:X\to Y$ be a fibration, $X\in C, a(Y)=0$. If a smooth fibre of $f$ is simple, then $Var(f)=0$, i.e: the general fibres of $f$ are bimeromorphic. 
\end{conjecture}

The first part of this conjecture thus holds when $n=2,3$, and would follow for any $n$ from a K\"ahler version of the Minimal Model Program, and the Bogomolov-Beauville decomposition (\cite{Bo}, \cite{B}) for compact K\"ahler terminal varieties with trivial canonical class (\cite{BGL}). 

\begin{remark}\label{r str si} Conjecture \ref{c simple}.1 implies that a strictly simple manifold is either an \'etale quotient of a torus, or hyperk\"ahler, and reduces in this case to showing that $\kappa\neq -\infty$ by the Bogomolov-Beauville decomposition.
\end{remark}

We establish Conjecture \ref{c simple}.1 for strictly simple fourfolds in \S\ref{ss 4f}, after some general remarks on the strictly simple case in \S\ref{ss}.

\begin{theorem}\label{t ss 4f} Let $X_4$ be a compact strictly simple K\"ahler manifold of dimension $4$. Then $X$ is either an \'etale quotient of a simple torus, or irreducible hyperk\"ahler.
\end{theorem}

In \S\ref{s ss} and \S\ref{s k}, we present a short survey of the semi-simple and Kummer reductions, after \cite{F82} possibly difficult to consult.  They are used in section \S\ref{s a=0}, devoted to the conditional description of manifolds of algebraic dimension zero in C.

\begin{remark}\label{r var ss} 1. If $X_y$ is strictly simple instead of just simple, Conjecture \ref{c simple}.2 follows from Conjecture \ref{c simple}.1 since $X_y$ is then either Hyperk\"ahler or covered by a torus. In both cases, $Var(f)=0$ is known (the first case by local Torelli, the second by period maps). 

2. See Remark \ref{ex var=0} for some other known examples where $Var(f)=0$. 
\end{remark}

When $X_y$ is irreducible hyperk\"ahler non-projective, Conjecture \ref{c simple}.2 is an easy consequence of the local Torelli theorem for holomorphic $2$-forms. However, it does not seem to immediately follow from the known versions of the local Torelli theorem for `primitive symplectic varieties' of \cite{BL}, unless one had a construction of terminal minimal models in families. Alternatively, a construction in the K\"ahler context of the Viehweg-Zuo sheaves for K\"ahler fibrations could possibly solve Conjecture \ref{c simple}.2 without solving first Conjecture \ref{c simple}.1.

\begin{theorem}\label{t' a=0} Let $X$ in C such that $a(X)=0$. Assume Conjecture \ref{c simple}.2. Let $\sigma_X:X\to \Sigma$ be the MRC\footnote{It is also the `algebraic coreduction' of $X$, introduced in \cite{Ca81}.} of $X$.  Its fibres are unirational prehomogeneous, and $\Sigma=(S\times K)/G$, with $K$ Kummer, $S$ a product of non-Kummer simple manifolds, and $G$ a finite group of automorphisms. If one also assumes  Conjecture \ref{c simple}.1, $S$ is isogeneous to a product of bimeromorphically symplectic irreducible manifolds.
\end{theorem}

Said otherwise: $\sigma:X\to \Sigma$ is the MRC of $X$, $\kappa(\Sigma)=a(\Sigma)=0$, and $\Sigma$ has a bimeromorphic Bogomolov decomposition.

\begin{example} The fibration $\sigma:X\to \Sigma$ occurs in general with $X\neq \Sigma$. For example, if $S$ is a `general' $K3$ surface with $Pic(S)=\{\mathcal{O}_S\}$, and if $TS$ is its tangent bundle, the only non trivial subvarieties of $\Bbb P(TS)$ are the fibres of  $a': \Bbb P(TS)\to S$. In particular, $a(\Bbb P(TS))=0$, and $a'=\sigma$. This remains true for any hyperk\" ahler manifold $S$ with $Pic(S)=\{\mathcal{O}_S\}$, as independently proved by Matei Toma (\cite{MP}, Appendix B). 
\end{example}





\section{Remarks on the strictly simple case.}\label{ss}

The strictly simple case is considerably easier to deal with than the general simple case for two reasons: $X$ is then a smooth minimal model (with $K_X$ nef, no singularities). 
The observations below are essentially those made in \cite{CDV}, with the replacement of \cite{Br} by \cite{O} for $n\geq 4$.

\begin{theorem}\label{strict simple} Assume $X$, compact K\"ahler smooth, $n$-dimensional, is strictly simple (i.e: it has no nontrivial subvariety). Then: 

1. $h^0(X,\Omega^p_X\otimes m.K_X)\geq  h^{n-p}(X,(m+1)K_X), \forall p\geq 0,m\geq 0$. 

Thus: $0=\chi(X,(m+1).K_X)$ if $h^0(X,\Omega^p_X\otimes m.K_X)=0$, for some $p\geq 0,m_0\geq 0$ and all $m\geq m_0$.

2. Either $K_X$ is torsion (case a) , or not, and:

 $\chi(X,m.K_X)=\chi(\mathcal{O}_X)=0, \forall m\in \Bbb Z$ (case b).
 
 In case b: $h^0(X,\Omega^p)>0$ for some odd $p>0$.
 
 In case a, $X$ is either irreducible hyperk\"ahler, or an \'etale quotient of a simple torus.
  \end{theorem}
  
  \begin{corollary}\label{c chi neq 0} Let $X_n$ be a strictly simple compact K\"ahler manifold. 
  
  1. If $\chi(X,\mathcal{O}_X)\neq 0$, $X$ is irreducible Hyperk\"ahler.
 
 2. If $n=3$, $X$ is an \'etale quotient of a simple torus (\cite{CDV}).
 \end{corollary}
 
 \begin{proof}(of Corollary \ref{c chi neq 0}, assuming Theorem \ref{strict simple}) Claim 1 follows immediately from \cite{B}, since we are in case a of Theorem \ref{strict simple}.
 
 Claim 2. If $p=3$ in case 2, $K_X$ is torsion since $n=3$, and the claim follows from Bogomolov decomposition (\cite{B}, \cite{Bo}). If $p=1$ in case 2, the Albanese map of $X$ is nontrivial, hence surjective and connected (by \cite{U}, 13.7). Since $X$ is simple, it must be bimeromorphic, hence isomorphic since strictly simple.
 \end{proof}

\begin{proof} (of Theorem \ref{strict simple}) By \cite{O} (see also \cite{CP} for a more direct proof), $X$ is not uniruled, hence $K_X$ is pseudoeffective. We equip $mK_X,m>0$ with a hermitian metric $h_m$ with a positive curvature in the sense of currents. Siu's decomposition theorem (\cite{S}) implies that the positive  Lelong numbers of the corresponding psh function are concentrated on points. From \cite{D}, Corollary 6.4, since $K_X$ is not big one concludes that all Lelong numbers vanish, and so that $\forall m>0, m.K_X$ is nef, and that moreover, its multiplier ideal sheaf $I(h_m)=\mathcal{O}_X$, this for every $m>0$. 

Let $w$ be a K\"ahler form on $X$. Takegoshi's hard Lefschetz theorem (\cite{T}), implies that, for each $p\geq 0$, the map:

 $L^{n-p}:=\bullet \wedge w^{n-p}:H^0(X,\Omega^p_X\otimes (mK_X))\to H^{n-p}(X,(m+1).K_X)$ is surjective.
Hence Claim 1.

\medskip

Claim 2. There  are two exclusive cases:

a. $H^0(X,\Omega^p_X\otimes (m.K_X))\neq 0$, for some $p\geq 0$, infinitely many $m>0$.

b. $H^0(X,\Omega^p_X\otimes (m.K_X))=0,\forall p\geq 0, m\geq m_0$, some $m_0>0$. 

\medskip

In the first case, we have $H^0(X,m.K_X)\neq 0$ for some $m>0$, by Lemma \ref{kappa=0} below, since $X$ being simple, we have: $a(X)=0$. Hence $m.K_X$ is trivial, and $K_X$ is torsion, since $X$ has in particular no divisor. 

\medskip

In the second case, by Takegoshi's theorem: $H^q(X,m.K_X)=0, \forall m\geq m_0$. The polynom $P(m):=\chi(X,m.K_X)$ thus vanishes for each $m\geq m_0$, and so identically. In particular $P(0)=\chi(X,\mathcal{O}_X)=0$. 

\medskip

We thus get either case a or case b.
\end{proof}

\begin{lemma}\label{kappa=0} (\cite{D-Pe-S}, Proposition 2.15) Let $X$ be a complex compact manifold, $E$ a vector bundle on $X$, and $L$ a line bundle on $X$. Assume that $H^0(X,E\otimes (m.L))\neq 0$ for infinitely many $m>0$. Then either $a(X)\geq 1$, or $\kappa(X,L)\geq 0$. 
\end{lemma} 

\begin{proof} We give a proof for completeness. Let $m_0,m_1,\dots, m_r$ be  pairwise distinct integers such that $\exists  s_j\in H^0(X,Hom(-m_j.L,E)), s_j\neq 0$ such that, if $F\subset E$ is the subsheaf  of $E$ generated by the images of the $s_j's$, its rank is $r$ (thus one less than the number of the $s_j's$). We thus obtain, possibly renumbering the $s_j's$, two non-zero maps $d:-(\sum_{j=0}^{j=r-1}m_j).L\to det(F)$, and $d':-(\sum_{j=1}^{j=r}m_j).L\to det(F)$. If $X$ (as in our situation) does not contain any effective nonzero divisor, these maps are isomorphisms, and  we obtain an isomorphism: $d'.d^{-1}:(m_r-m_0).L\to \mathcal{O}_X$. If, more generally, $a(X)=0$, then $X$ contains only finitely many irreducible effective divisors by \cite{K}, and we conclude in the same way, increasing the number of sections by the number of such irreducible divisors.
\end{proof}

Of course, case 2 above should be excluded (it implies that $\kappa(X)=-\infty$). For this, it is sufficient to show that $L=K_X$ is numerically trivial, i.e: that its numerical dimension $\nu(X,L)$ vanishes, where $\nu(X,L)$ is the largest integer $k\geq 0$ such that the $k$-th self-intersection $L^k=K_X^k$ is not numerically zero, or equivalently such that $K_X^k.w^{n-k}>0$ for some/any K\"ahler form $w$ on $X$, since $K_X$ is nef.

One can bound $\nu(X,K_X)$ from above by the following Kawamata-Viehweg type statement (special case  $I(h)=\mathcal{O}_X$ of loc. cit.) :

\begin{lemma}\label{K-V-T} (\cite{T}, Theorem 4.2 (i)) Let $L$ be a nef line bundle on a compact K\"ahler manifold $X$ of dimension $n$, $L$ equipped with a hermitian metric $h$ with $I(h)=\mathcal{O}_X$. Let  $\nu:=\nu(X,L)$.

Then: $H^0(X,\Omega^q_X\otimes L)=H^{n-p}(X,K_X+L)=0, \forall q\geq (n-\nu+1)$.

\medskip

In particular, if $X$ is strictly simple, and if $H^0(X,\Omega^p\otimes (m.K_X))\neq 0$, for some $m>0$, then $\nu(X,L)\leq p$. 

\medskip

For $p=1$, this means that $\nu(X,K_X)\leq 1$ if $H^0(X, \Omega^1_X\otimes K_X)\neq 0$.
\end{lemma}





\section{Strictly simple fourfolds.}\label{ss 4f}

\begin{theorem}\label{t 4f}Let $X$ be a strictly simple compact K\"ahler fourfold. Then $X$ is either hyperk\"ahler, or an \'etale quotient of a simple torus.
\end{theorem}

\begin{proof} We apply Theorem \ref{strict simple}. In case 1, we get the conclusion from Bogomolov decomposition (\cite{B}).
 We thus assume that we are in case 2. We assume moreover that $q:=h^0(X,\Omega^1_X)=0$ (resp. $h^0(X,K_X)=0$), since otherwise $X$ is isomorphic to its Albanese torus (resp. $K_X$ is trivial).

 We shall prove the following lemma \ref{lk torsion}, which obviously implies Theorem \ref{t 4f}:
 
 \begin{lemma} \label{lk torsion} Let $X$ be a compact K\"ahler smooth fourfold such that $\chi(X,\mathcal{O}_X)=0$. If $X$ does not contain effective divisors or surfaces, $K_X$ is torsion if pseudoeffective\footnote{One may wonder if this last condition is superfluous. A counterexemple were provided by the projectivisation $\Bbb P(E)$ for a rank-$2$ vectorbundle $E$ over a simple non-projective torus $T$ of dimension $3$ such that $h^0(T',f^*(E))=0, \forall T'\to T$ \'etale cover. Such $E$'s do exist on non-projective simple tori of dimension $2$ (see\cite{To}, based on the Serre construction).}.
 \end{lemma} 
 
 Notice that $a(X)=0$.
 
 From Riemann-Roch, Hodge decomposition and $\chi(X,\mathcal{O}_X)=0$, we then get: $h^0(X,\Omega^3_X)=1+h^0(X,\Omega^2_X)$, and $h^0(X,\Omega^2_X)>0$ since $X$ is not projective, by Kodaira's criterion (\cite{K}).
 
 In particular, we have (at least) one (resp. two) non-zero holomorphic two-form $\sigma$ (resp. $3$-forms $\tau_i,i=1,2$ which are non-proportional).
 
 We suppose that $\sigma$ has generic rank $2$ (i.e: $\sigma^{\wedge 2}=0$) otherwise $K_X$ is trivial and the claim follows. 
 
 Let $F=\sigma^{\perp}\subset TX$ be the (rank-$2$) foliation defined by $\sigma$, regular ouside a finite set of curves. The locally free quotient $TX/F$ is equipped with a symplectic $2$-form deduced from $\sigma$. In particular: $TX/F$ is self-dual, $det(TX/F)=\mathcal{O}_X$, and $det(F)=-K_X$. 
 
 Similarly, let $G_i:=\tau_i^{\perp}\subset TX$ be the rank-one foliation defined by $\tau_i$ for $i=1,2$. Then $TX/G_i$ is equipped with a non-vanishing $3$-form, and so $det(TX/G_i)=\mathcal{O}_X$, and $det(G_i)=-K_X, i=1,2$. The $G_i's$ may have singularities at a finite set of curves.

 Since $a(X)=0$ the foliations $G_i$ are different (otherwise $\tau_1$ and $\tau_2$ were meromorphically proportional, hence linearly proportional).
 
 Thus $G_1\cap G_2=\{0\}$ generically, hence if $G:=G_1+G_2$, we get (since $X$ does not contain any effective divisor): $det(G)=-2K_X$, and $det(TX/G)=-K_X-(-2K_X)=+K_X$.
 
 Let now $J:=G\cap F$, and let $r\in \{0,1,2\}$ be its rank.
 
 \medskip
 
 If $r=0$, $TX=G+F$ generically, hence $-K_X=det(TX)=det(G)+det(F)=-2K_X-K_X=-3K_X$. Hence $2K_X=\mathcal{O}_X$.
 
 \medskip
 
 If $r=2$, $F=G$ generically, hence $-K_X=det(F)=det(G)=-2K_X$, and $K_X=\mathcal{O}_X$.
 
 \medskip
 
 Assume now that $r=1$. Let $L:=det(J)$. Let $p_i:G\to G/G_i$ be the projection, $i=1,2$. Not both of its restrictions to $L$ are zero, say it is not zero for $i=1$. Hence $L=det(G/G_1)=-K_X$ (again because $X$ does not contain any nonzero effective divisor). Thus:
 
 (1) $det(J)=-K_X$

 Thus $det(TX/J)=det(TX)-det(J)=-K_X-(-K_X)=\mathcal{O}_X$. We have a natural (generically surjective map: $TX/J\to TX/F$ with kernel $F/J$. Since $det(TX/J)=det(TX/F)=\mathcal{O}_X$, we get: $\mathcal{O}_X=det(F/J)$, and thus an exact sequence: 
 
 (2) $0\to -K_X\to F\to \mathcal{O}_X\to 0$.
 
 This sequence does not split, since otherwise we had injections $\mathcal{O}_X\subset F\subset T_X$, and thus $dim(Aut^0(X))>0$, which would imply, by \cite{L}, Theorem 3.12, that either $X$ is uniruled, or that $q(X)>0$, and thus that $X$ were a (simple) torus, contradicting our hypothesis. From this non-splitting, we deduce that $F$ is stable. Let indeed $H\subset F$ be a rank-one saturated subsheaf. If the natural map $H\to F/J$ is not zero, it is a isomorphism which splits the exact sequence (1). Contradiction. Thus $H=J$. Since $\kappa(X)=-\infty$ by assumption, $K_X$ is not numerically trivial, $K_X.w^{3}>0$, for any K\"ahler form $w$ on $X$, so that $F$ is $w$-stable.

 From the exact sequence (2), we deduce that $c_2(F)=0$ since the sequence is exact outside finitely many curves. Let $w$ be any K\"ahler class on $X$. 
 
 From stability, L\"ubke inequality reads\footnote{This argument is inspired by the same argument (applied to $G$ instead of $F$ in the proof of \cite{HPR}, Theorem 4.4), and replaces the argument based on Takegoshi's K\"ahler version of the Kawamata-Viehweg theorem in the original proof.}:
 
  $((r-1)c_1(F)^2-2r.c_2(F)).w^2=-K_X^2.w^2\leq 0$, and \cite{HPR}, Lemma 4.3 (according to which $K_X^2.w^2\geq 0$ because $K_X$ is pseudoeffective, and there are no subvariety of codimension one or two),  we deduce that $K_X^2.w^2=0$, and so $K_X^2$ is numerically trivial.

We shall now conclude following Miyaoka's strategy in \cite{Miy}, later used in \cite{Ko}.

Consider the bundle $E:=F^*\otimes F=End(F)$. It is polystable, and has $c_1(E)=0, c_2(E)=0$, the last equality because $K_X^2\equiv 0$. It thus arises, by fundamental results of Donaldson, Uhlenbeck-Yau (\cite{Do}, \cite{UY}, \cite{Kob}, proposition 4.4.13, strengthened in \cite{BS})\footnote{I thank Matei Toma for clarifications and the last reference.} from a unitary representation $\rho:\pi_1(X)\to U(4)$. 
There are two cases:

A. The image of $\rho$ is finite. Replacing $X$ by a finite \'etale cover, we may assume that $E$ is a trivial bundle. We thus have: $h^0(X,E)=4$. On the other hand, we have an exact sequence:

(6) $0\to F^*\otimes (-K_X)\to E\to F^*\otimes K_X)\to 0$, and so:

$E$ has thus an increasing filtration with rank-one graded pieces:

 $-K_X\oplus \mathcal{O}_X\oplus -K_X\oplus K_X,$
 so that $h^0(X,E)\leq 1$, since $0=h^0(X, m.K_X)$, for $m=\pm 1$, by assumption. Hence a contradiction. 
 Thus case A does not occur.

B. The image of $\rho$ is infinite. It then follows from \cite{CCE}, Theorem 0.3, that the image of $\rho$ is virtually abelian (as is always the case for $X$ compact K\"ahler with $a(X)=0$, or more generally: $X$ `weakly special', meaning that no finite \'etale cover of $X$ maps surjectively onto a manifold of general type and positive dimension). Thus $X$ is an \'etale quotient of a simple torus, and $K_X$ is torsion, contradicting our assumptions.\end{proof}
 
 \begin{remark} 1. Proving the general fourfold simple case seems to be considerably more difficult, involving a K\"ahler version of the MMP, as seen from the threefold case (compare \cite{CDV} and \cite{CHP}).
 
 2. After this text appeared on arXiv, A. H\"oring informed me of the text \cite{HPR}, especially Theorem 4.4, in which it is proved that a compact K\"ahler fourfold without divisors and surfaces has an \'etale cover which is a torus if it is uniformised by $\Bbb C^4$. As said above, one of their arguments (L\"ubke inequality) permits to simplify our original proof.
 
 \end{remark}

\section{Semi-simple and Kummer reductions.}\label{s ss}

We present here a shortened and simplified exposition of \cite{F82}, since this text may be difficult to find. See also \cite{COP}, \S 2. 

\begin{definition} Let $X\in C$. A covering family $(Z_t)_{t\in T}$ of $X$ is an irreducible analytic subset $Z\subset T\times X$ for some compact irreducible analytic subset with surjective projections $p:Z\to T,q:Z\to X$ such that:

1. the generic fibre $Z_t$ of $p$ is irreducible.

2. the cycle map map $\zeta:T\to BC(X)$ sending $t$ to $Z_t$ is generically injective.

The covering family $(Z_t)_{t\in T}$ is `maximal' if, moreover:

3. For $t\in T$ `general', the only subvariety of $X$ containing $Z_t$ is $X$.

4. For $t\in T$ generic, $Z_t$ is of codimension at least $2$ in $X$ (divisors are excluded).
\end{definition}

\begin{lemma}\label{l sigma} Let $(Z_t)_{t\in T}$ be a maximal covering family of $X$. Then:

1. $q:Z\to X$ is generically finite (of degree $d$ say).

2. $T$ is simple.

Let $\sigma:X\to Sym^d(T)$ be the map sending a generic $x\in X$ to $p(q^{-1}(x))$. Let $\Sigma:=\sigma(X)\subset Sym^d(T)$. Then:

3. $\Sigma$ is semi-simple, `isogeneous' to $T^{\delta}$, for some $0\leq \delta\leq d$. 

Recall that `isogeneous' means that some irreducible compact $\Sigma'\subset \Sigma\times T^r$ exists, generically finite over both $\Sigma$  and $T^r$ for the natural projections.
\end{lemma}

An immediate consequence is:

\begin{corollary}\label{c sigma} If $a(X)=0$, $X$ is either simple, or admits a maximal covering family, and has thus a fibration $\sigma:X\to \Sigma$ for some positive-dimensional isotypically semi-simple $\Sigma$.
\end{corollary}

\begin{proof} (of Lemma \ref{l sigma}) Claim 1. By \cite{Ca81}, the fibres of $q:Z\to X$ are Moishezon. If they are positive-dimensional, they are covered by curves. And there is thus a covering family $(C_s)_{s\in S}$ of $T$ by curves. The family $W_s:=q(p^{-1}(C_s))$ is thus a covering family of $X$ by cycles of dimension $dim(Z_s)+1$, each containing some $Z_t$, contradicting the maximality of the family $(Z_t)_{t\in T}$. Thus $dim(Z)=dim(X)$ and $q$ is generically finite.

2. If $T$ is not simple, it admits a covering family by cycles $(V_s)_{s\in S}$ of positive dimension less than $dim(T)$. The cycles $W_s:=q(p^{-1}(V_s)), s\in S$ thus form a covering family of $X$ of cycles of dimension  $dim(V_s)+dim(Z_t)>dim(Z_t)$ since $q:Z\to X$ is generically finite. This contradicts the maximality of the family $(Z_t)_{t\in T}$.

3. Let $\pi:T^d\to Sym^d(T)$ be the natural finite projection, and $\Sigma'\subset T^d$ be an irreducible component of $\pi^{-1}(\Sigma)$, surjective over $\Sigma$. The natural projections $p_j:\Sigma'\to T,j=1,\dots,d$, are surjective since $Z$ is irreducible, surjective on $T$. The conclusion then follows from the next Lemma \ref{l ssimple}.
\end{proof}

\begin{lemma}\label{l ssimple} Let $T$ be simple, and $\Sigma'\subset T^d$, irreducible compact, surjective onto each factor. There then exists $K\subset \{1,\dots,d\}$ such that the projection $p_K:\Sigma'\to T^K$ is surjective, generically finite, where $T^K:=T^{k_1}\times\dots\times T^{k_r}$, for $K:=\{k_1,\dots,k_r\}$. 
\end{lemma}

\begin{proof} Choose $K\subset \{1,\dots,d\}$ maximal such that $p_K:\Sigma'\to T^K$ is surjective. If $K=\{1,\dots,d\}$, $\Sigma'=T^d$. Otherwise, let $j\notin K$, and let $L:=K\cup j$. Then $p_L:\Sigma'\to T^L$ is not surjective. For $y\in T^K$ generic, $p_j(p_K^{-1}(y))\subsetneq T$ is a covering family of (not necessarily irreducible) analytic subsets. They are thus finite sets since $T$ is simple. The map $p_K:\Sigma'\to T^K$ is thus generically finite, and $\Sigma'$ is semi-simple, isogeneous to $T^d$, as claimed.
\end{proof}

\begin{theorem}\label{red ss}(\cite{F82}) Let $X\in C$ with $a(X)=0$. There exists a unique fibration $\sigma_X:X\to S_X$, with $S$ semi-simple (positive-dimensional) such that for each such fibration $\sigma:X\to S'$, there is a unique factorisation $\tau: \Sigma_X\to \Sigma'$ such that $\sigma'=\tau\circ \sigma_X$. 

There is also a relative version: if $f:X\to Y$ is a fibration in C, there is a unique fibration $\sigma_f:X\to S_f, \sigma'_f:S_f\to Y$ such that $f=\sigma'_f\circ \sigma_f$, and inducing $\sigma_{X_y}: X_y\to S_y$ on the general fibre $X_y$ of $f$. 

Then $\sigma_X$ (resp. $\sigma_f$) is the semi-simple reduction of $X$ (resp. of $f$).

If $dim(X)=0$, $dim(S_X)>0$.
\end{theorem}

\begin{proof}Let $\sigma: X\to S$ be a fibration onto a semi-simple manifold $S$ of maximum dimension. Let $\sigma':X\to S'$ be another such fibration. Assume that $\sigma$ does not factorise through $\sigma'$. Let $\sigma"=\sigma\times \sigma':X\to S\times S'$, and let $S"\subset S\times S'$ be its image. Let finally $s:X\to \Sigma, u:\Sigma\to S", \sigma"=u\circ s$ be the Stein factorisation of $\sigma"$. Since $dim(S")>dim(S)$ by the non-factorisation hypothesis, the conclusion of the first Claim follows from the next lemma \ref{l sss}.

The second Claim follows from the first and from \cite{Ca04}, appendix, once the Zariski-regularity of the properties simple and semi-simple are established (in Lemma \ref{l Zreg} below).

The last Claim follows from Corollary \ref{c sigma}.
\end{proof}

\begin{lemma}\label{l sss} Let $S:=S_1\times \dots\times S_r$ be semi-simple manifolds in C, with each $S_i=T_i^{d_i}$, the $T_i$'s being simple manifolds, $T_i$ and $T_j$ non-isogeneous if $i\neq j$. 
Let $\Sigma\subset S$ be a subvariety, and $\Sigma_i\subset S_i$ be the image of the $i$-th projection $p_i:S\to S_i$. Assume that each $\Sigma_i$ is surjectively mapped onto $T_i$, for each factor of $T_i^{d_i}$ . Then:

1. $\Sigma=\Sigma_1\times \dots \Sigma_r$. 

2. $\Sigma$ is semi-simple. 

3. If $\Sigma'\subset S_1\times S_2$, where $S_i,S_2$ are semi-simple, and if $\Sigma$ is mapped surjectively onto $S_1$ and $S_2$ by the natural projections, then $\Sigma$ is semi-simple.
\end{lemma}

\begin{proof} Claim 1. Each $\Sigma_i$ is isogeneous to some $T_i^{K_i}$, after Lemma \ref{l ssimple}. We can thus assume that $\Sigma_i=S_i$, for each $i$. Let $T_{i,k}$ be the $k$-th factor of $T_i^{d_i}$, for $1\leq k\leq d_i$, and $ \hat{S}_{j,\ell}$ the product of all factors of $S$, omitting the single factor $T_{j,\ell}$. If $\Sigma\subsetneq S$, there are some $i<j$ and factors $T_{i,k}, T_{j,\ell}$ of $S_i, S_j$ respectively such that, for general $y'\in \hat{S}_{j,\ell}$, the restriction to $(S\cap (T_j\times \{y'\}))$ of the projection of $S$ onto $T_{i,k}$ is not surjective, hence generically finite and so induces an isogeny between $T_j$ and $T_i$, contradicting our hypotheses.

Claim 2. By Lemma \ref{l ssimple}, each $\Sigma_i$ is semi-simple. Hence so is their product.

Claim 3. $S,S'$ are isogeneous to products of simple manifolds $T_i$, one can thus replace them by these products, so that $\Sigma$ is contained in a product of $T_i's$, and surjective on each factor. It is then sufficient to group the $T_i's$ by isogeneity types, and apply Claims 1 and 2. 
\end{proof}

\begin{remark} One easily checks that the proofs of lemmas \ref{l ssimple} and \ref{l sss} remain valid with semi-simple replaced by Kummer of algebraic dimension zero, and simple $T_i$'s by complex (compact) tori of algebraic dimension zero, since holomorphic maps between tori are affine, and isogeneity between complex tori is realised by \'etale covers. From this one deduces the following analogue \ref{red k} of Theorem \ref{red ss}. The first two Claims of the next result are proved in (\cite{F83},\S7). The third Claim is a consequence of (\cite{F83}, Remark 7.1).
\end{remark}

\begin{theorem} \label{red k}  1. Let $X\in C$ with $a(X)=0$. There exists a unique fibration $k_X:X\to K_X$, with $K$ Kummer such that for each such fibration $k':X\to K'$, there is a unique factorisation $t: K_X\to K'$ such that $k'=t\circ k_X$. 

2. There is also a relative version: if $f:X\to Y$ is a fibration in C, with $a(f)=0$ (meaning that $a(X_y)=0$ for $y\in Y$ general), there is a unique fibration $k_f:X\to K_f, u:K_f\to Y$ such that $f=u\circ k_f$, inducing $k_{X_y}: X_y\to K_{f,y}=K_{X_y}$ on the general fibre $X_y$ of $f$. 

Then $k_X$ (resp. $k_f$) is the Kummer reduction of $X$ (resp. of $f$).

3. Let $f:X\to Y$ be as above, $g:X'\to X$ be generically finite surjective, with $f\circ g=g'\circ f', f':X'\to Y', g':Y'\to Y$ the Stein factorisation of $f\circ g$. There is a natural map $K(f,f'):K_{f'}\to K_f$ such that $k_f\circ K(f,f')=g'\circ K_{f'}:K_{f'}\to Y$, and $K(f,f')$ is generically finite surjective. In particular: $k(f'):=k(X'/Y')=k(X/Y)=k(f)$. 

\end{theorem}

\begin{remark} It is easy to check that the semi-simple reduction of a complex torus $T$ with $a(T)=0$ is an affine fibration $\sigma: T\to T'$, where $T'$ is isogenous to a product of simple tori (of algebraic dimension zero). Its Kummer reduction is of course its identity map.
\end{remark} 

\begin{lemma}\label{l Zreg} Let $f:X\to Y$ be a fibration.

1. The general $X_y$ is simple (resp. semi-simple) if some $X_y$ is.

2. If the general $X_y$ is semi-simple, there is a generically finite $u:X'\to X$ such that, if $f':X'\to Y', v:Y'\to Y$ is the Stein factorisation of $f\circ u=v\circ f'$, then $X'=T_1\times_{Y'} \dots\times _{Y'} T_{r-1}\times_{Y'} T_r$, where $\tau_i:T_i\to Y', i=1,\dots, r$ has general fibres simple.
\end{lemma} 

\begin{proof} We sketch the proof, up to standard technicalities.

1. The assertion means (for the simple case) that there is no $\tau: T \to Y$, component of $BC(X/Y)$ which induces on the general fibre $X_y$ of $f$ a maximal covering family of $BC(X_y)$ of cycles of codimension at most $2$. In the semi-simple case, it means that the intersection through a general point $x$ of $X$ of the cycles parametrised by  such $\tau: T\to Y$ relatively covering maximal families intersect in a finite set.

2. Let $\tau_j: T\to Y, j\in \{1,\dots, r\},$ be a finite set of such relatively maximal and covering family of cycles in $BC(X/Y)$: its general member $T_y$ is thus simple by Lemma \ref{l sigma}. Moreover, for each $j$, if $Z_j\subset T_j\times X$ is the incidence graph of this family, the projection $q_j: Z_j\to X$ is surjective and generically finite of degree $\delta_j$. We thus define, (as in the proof of Lemma \ref{l sigma}), a map $\sigma_j:X\to Sym^{\delta_j}(T_j)$ over $Y$ by: $\sigma_j(x)=(p_j)_*(q_j^{-1}(x)), p_j:Z_j\to T_j$ being the projection on the first factor. The map $\sigma_j$ is over $Y$ defined by: $\sigma_j(X_y)\subset Sym^{\delta_j}(T_{j,y})$, with $T_{j,y}:=(f_*)^{-1}(y)$, $f_*:BC(X/Y)\to Y$ being the natural projection map which sends to $y$ a cycle in $X_y$.

The general fibre of $f$ is semi-simple if and only if we can choose enough of the families $T_j$ such that the map $\sigma:=\sigma_1\times\dots \times \sigma_r:X\to Sym^{\delta_1}(T_1)\times _Y\dots _Y\times Sym^{\delta_r}( T_r):=S$ is generically finite onto its image (in general, this map defines the semi-simple reduction of $f$ by choosing the dimension of the image maximum). We then conclude the proof by taking a component $X'$ surjective on $X$ of the inverse image in $X\times T$ of the graph of $X"\subset X\times S$ of $\sigma$ under the finite projection $T:=T_1^{\delta_1}\times _Y\dots _Y\times T^{\delta_r}\to Sym^{\delta_1}(T_1)\times _Y\dots _Y\times Sym^{\delta_r}(T_r)$. Such a component $X'$ is indeed generically finite on $X$, equipped with projections on the $T_i^{\delta_i}$ compatible with those of $X$ on the $Sym^{\delta_i}$. 
\end{proof}

\begin{proposition}\label{p ks}(\cite{F83}, 10.2) Let $a':X\to A, a:A\to K, k:K\to Y$ be a factorisation of $f=k\circ a\circ a'$, with $a(f)=0$ and $k:K=K(f)\to Y$ the Kummer reduction of $f$\footnote{Or more generally, a fibration such that $k(f)=k(k)$.}, and $a:A\to K$ the algebraic reduction of $a\circ s:X\to K$. Then: $q(a)=0$, and $a(a')=k(a')=0$. 
\end{proposition}

\begin{corollary}\label{c ks} 1. Let $f:X\to A\to K\to Y$ be as in Proposition \ref{p ks}. The fibres of $a$ are then unirational prehomogeneous. Moreover, if $s'_{a'}:X\to S, s_{a'}: S\to A$ is the semi-simple reduction of $a'=s_{a'}\circ s'_{a'}:X\to A$, the general fibre of $s_{a'}:S\to A$ is semi-simple without Kummer factor.

2. Let $g:X'\to X$ be generically finite onto. Let $f:X\to S\to A\to K\to Y$ be as in the preceding situation of Corollary \ref{c ks}.1. 

Let $f':X'\to Y', g':Y'\to Y$ be the Stein factorisation of $f\circ g=g'\circ f'$. 

Let $f':X'\to S'\to A'\to K'\to Y'$ be the decomposition of $f':X'\to Y'$ analogous to the one for $f$ of the present Corollary \ref{c ks}.1. 

There are then natural maps, all generically finite onto: $K'\to K, A'\to A, S'\to S$ making the natural diagram relative to the factorisations of $f,f'$ commutative. 

\end{corollary}

\begin{proof} The first Claim is a direct consequence of Proposition \ref{p ks}, since the fibres of $A\to K$ are prehomogeneous with $q=0$. Since $a(X/A)=0=k(X/A)$, the semi-simple reduction of $X/A$ exists and has no Kummer factor.

The existence and surjectivity of the maps $K'/K, A'/A, S'/S$ as well as the commutativity of the diagram are easy. The fact that $K'/K$ is generically finite onto is Theorem \ref{red k}.3. The fact that $A'/A$ and $S'/S$ are generically finite onto then just follows from the fact that finite covers of Moishezon (resp. semi-simple manifolds without Kummer factor) enjoy this same property, this fact being applied to the general fibres of $A\to K$ and $S\to A$. 
\end{proof}





\section{Remarks on simple Kummer manifolds.}\label{s k}

Let $T$ is a complex torus, and $G$ a finite group of holomorphic automorphisms of $T$ (operating faithfully), and $X$ a smooth model of $T/G$. Let $n:=dim(T)=dim(X)$. Notice that $X$ and $T/G$ have the same fundamental group, and that $H^{p,0}(X)=H^{p,0}(T)^G$ by \cite{Fr}.

Recall that a complex torus $T$ is t-simple if it does not contain any nontrivial subtorus, and simple if it does not contain any nontrivial subvariety (and if its dimension is $n\geq 2$). Thus simple implies t-simple. Conversely, if $T$ is t-simple, it is simple if and only if either non-projective, or, equivalently,  if $a(T)=0$ (the algebraic reduction being a quotient by a subtorus, for any $T$). 

If $T$ is t-simple, and if $g\neq 1$ is an automorphism of finite order $m$ of $T$, the fixpoints of $g$, if any, are isolated, and if $y$ is a fixpoint of $g$, the eigenvalues of $g$ on $T_y$, the tangent space of $T$ at $y$, are all primitive $m$-th roots of unity (\cite{BL'}, \S 13.2, most of it  valid for tori not only Abelian varieties.).

\begin{lemma}\label{can}  If $T$ is not an Abelian variety, and if the fixpoints of any $1\neq g\in G$ are isolated, the singularities of $T/G$ are canonical if $n=2$, and terminal if $n\geq 3$. In particular, $\kappa(X)=\kappa(T/G)=0$. More precisely: $H^0(T,\Omega^2_T)^G$ is generated by $2$-forms of rank $2$. 
\end{lemma}

\begin{proof} Since $T$ is not projective, so is $X$, and Kodaira's theorem (\cite{Kod}) implies that $h^0(X,\Omega^2_X)>0$. If $G$ acts without fixpoints, the quotient $\pi:T\to T/G$ is \'etale, hence $T/G$ is smooth, and $\kappa(T/G)=\kappa(T)=0$. 
Otherwise, we apply the Reid-Tai criterion. Let $y\in T$ be a fixpoint of some $g\neq 1$ in $G$. 

Let $m>1$ be the order of $g$, $exp(\frac{2\pi i.a_j}{m}), 0\leq a_j<m, j=1,\dots, m$ being the eigenvalues of $g$ acting on $T_y$, the tangent space of $T$ at $y$. The singularity of $T/G$ at $\pi(y)$ is canonical (resp. terminal) if $age(g):=\sum_ja_j\geq m$ (resp. $age(g)>m$). By a result of Freitag (\cite{Fr}) $H^0(X,\Omega^2_X)=H^0(T,\Omega^2_T)^G$. Let  $0\neq s\in H^0(A,\Omega^2_T)^G$. Write $s=\sum_{j<k}c_{jk}.dz_j\wedge dz_k$ in linear coordinates of $T_y$ diagonalizing $g$. Thus $g^*(s)=\Sigma_{j<k} exp(2\pi i.(a_j+a_k)).c_{jk}.dz_j\wedge dz_k$ and $m$ divides $a_j+a_k$ if $c_{jk}\neq 0$. Since the singularities of $g$ are isolated, $0<a_j, \forall j$. Thus $a_j+a_k=m$ if $c_{jk}\neq 0$. From which follows that $age(g)=m$ if $n=2$, and $age(g)>m$ if $n\geq 3$. Thus the conclusion. \end{proof}

We shall apply some of the arguments and results of of \cite{Ca20}. We denote with $\kappa_1(T/G)=\kappa_1(X)$ the Kodaira dimension in $\{2n-1,\dots, 0,-\infty\}$ of the tautological line bundle $\mathcal{O}_P(1)$ on $P:=\Bbb P(\Omega^1_X)$, which measures the rate of growth of $h^0(X,Sym^k(\Omega^1_X))$ as $k\geq 0$ tends to $+\infty$. 

\begin{lemma}\label{kappa_1} Let $T/G$ be as above. Let $g\in G$ be a generator of $G$ of order $m>1$ with an isolated fix point $y\in T$. Then:

1. $h^0(X,Sym^k(\Omega^1_X))=0, \forall k>0$.

2. $\pi_1(T/G)=\pi_1(X)$ is finite abelian if $G$ is cyclic, generated by $g$, and $T/G$ is simply-connected unless $m=p^k$ for some prime $p$, in which case $\pi_1(X)$ is a quotient of  $\Bbb Z_p^{\oplus 2n}, n:=dim(T)$.
\end{lemma}

\begin{proof} Claim 1 follows from the proof of Lemma 5.18  of \cite{Ca20} (first part $D=0$. Although the result is stated only for $T$ projective, the proof applies to a complex torus).

Claim 2 follows from Lemma 5.19 of loc. cit. (which applies to the non-projective case as well). 

In loc.cit. one proves that $\pi_*:\pi_1(T)\to \pi_1(T/G)$ is surjective with finite image. Thus $\pi_1(T/G)$ is a finite abelian group generated by $2n$ elements. We have $\pi_*\circ g_*=\pi_*$ for the action on the loops on $T$ based at a fixpoint of $g$. Since $g$ has isolated fixpoints, its eigenvalues are primitive $m$-th roots of the unity, and $\Phi_m(g)=0$ if $\Phi_m(z)=\sum_{k=0}^{k=\varphi(m)}a_k.z^k$ is the $m$-th cyclotomic polynomial. We then have: $0=\pi_*\circ (\sum_ka_k.(g^k)_*)=\sum_ka_k.\pi_*\circ (g^k)_*=\sum a_k.\pi_*=\pi_*.\Phi_m(1)$. The exponent of $\pi_1(T/G)$ thus divides $\Phi_m(1)$ which is $p$ if $m$ is a power of $p$, and is $1$ otherwise. This proves the claim.\end{proof}

The following result applies in particular to any $X$ simple Kummer:

\begin{corollary} 1. If $T$ is t-simple with $h^0(X,\Omega^2_X)\neq 0$, then $\kappa(X)=0$. 

In particular, $\kappa(X)=0$ if $T$ is simple.

If $X=T/G$ is not an \'etale quotient of $T$. Then:

2. $h^0(X,Sym^k(\Omega^1_X))=0,\forall k>0$, and:

3. If $G$ is cyclic of order $m$, $\pi_1(X)$ is finite abelian, quotient of  $\Bbb Z_p^{\oplus 2n}$ if $m$ is a power of a prime number $p$, and simply connected otherwise. 
\end{corollary}

\begin{example} 1. $(T/\pm 1)$, $n\geq 2$. Then $\kappa=0, p_g=1$ if $n$ is even, $p_g=0$ is $n$ is odd; $h^0(X,\Omega^p_X)=(^n_p)$ if $p$ is even, zero otherwise, and $\kappa_1=-\infty$, $\pi_1(X)=\{1\}$\footnote{This does not follow from the preceding corollary, but is easily seen from the fact that if $\gamma$ is the linear path joining two $2$-torsion points of $T$, its projection on $(T/\pm 1)$ is a loop equal to itself by $-1$.}. This is valid without the t-simpleness assumption. Except for $\kappa_1$, this is observed in \cite{U}, 16.11.

2. $T$ simple. What are the possible $(T,G)$ for each given $n$?
\end{example}

\begin{proposition}\label{k vs hk} Let $X_n=T/G$ be a Kummer variety with $T$ simple and $n=2m\geq 4$. Assume that $G$ is Abelian. If $h^{2,0}(X)=1$, a generator $s$ of $H^{2,0}(X)$ is of rank $2$ (i.e: $s^{\wedge 2}=0$).

In particular, $X$ is not bimeromorphically symplectic irreducible. 
\end{proposition}

\begin{proof} Let $z_i, i=1,\dots, 2m$ be linear coordinates in which the action of $G$ is diagonalised. Any $G$-invariant $2$-form $u$ on $T$ is written as: $u=\sum_{j<k} c_{jk} dz_j\wedge dz_k$, and $g^*(u)=\sum_{jk} c_{jk}.\chi(g).dz_j\wedge dz_k$, for the character associated to $g$ by the representation of $G$ on the tangent bundle of $T$. The vector space $H^{2,0}(X)$ is generated by the $G$-invariant $2$-forms (\cite{Fr}). Since $h^{2,0}(X)=1$, there is a single pair $\{j, k\}$ such that $c_{jk}\neq 0$ and $\chi(g)=1, \forall g\in G$. Its rank is of course $2$. \end{proof}

\begin{question}\label{q k} Does there exist simple Kummer manifolds $T/G$ of even dimension $n=2m\geq 4$ which are irreducible  bimeromorphically symplectic, i.e: such that $H^{2,0}(T)^G=H^{2,0}(T/G)=\Bbb C$, generated by a $2$-form $s$ such that $s^{\wedge m}\neq 0$? The preceding proposition shows that this is not possible if $G$ is abelian. One may ask a weaker version: does there exists  simple $T/G$ such that $H^{p,0}(T)^G=0$ for $p$ odd, and is equal to $\Bbb C$, generated by $s^{p/2}$ for $p$ even. The Kummer quotients $T/G$ with $G$ acting freely in codimension $2$-as in our `simple' situation-are characterised by the vanishing of $c_2$ (\cite{CGG}). The link with the distribution of the $h^{p,0}$ according to parity is however not immediate.
\end{question}





\section{Conditional description of $X$'s with $a(X)=0$.}\label{s a=0}

Let $f:X\to Y$ be a fibration, $X,Y$ smooth connected in C, $X_y$ the `general' fibre of $f$.

\subsection{Bimeromorphic variation of a fibration.}\label{s var}

Let $f:X\to Y$ be a fibration with $X,Y$ smooth, connected, compact, in the class C. If $b:B\to Y$ is a proper connected surjective, we denote with $f_B:(X\times_Y B)\to B$ the fibration deduced from $f$ by the base change $b$ on its main component.

\begin{definition} 1. We say that $Var(f)=0$ if there exists $b:B\to Y$ proper connected and $F \in C$ such that $X_B$ is bimeromorphic to $F\times B$ over $B$. 

2. $Var(f)$ is said to factorise through the algebraic reduction $g:Y \to Z$ of $Y$ if $Var(f_z)=0$, where $f_z:X_z\to Y_z$ is the restriction of $f$ to the general fibre of $g\circ f:X\to Z$. 

\end{definition} 

Such notions have been defined initially by Viehweg, Kawamata, Koll\'ar in the study of conjecture $C_{n,m}^+$ for projective manifolds, and motivated the construction of Viehweg-Zuo sheaves. 

We consider here K\"ahler non-projective situations in which $a(Y)=0$ might imply that $Var(f)=0$ under suitable conditions on the general fibre. It may be true that this always holds if the fibres of $f$ are not uniruled.  See Remark \ref{r var} below.

\begin{remark} Using standard arguments and the compactness of the components of $BC(X/Y)$, one shows that $Var(f)=0$ if and only if the general fibres of $f$ are pairwise bimeromorphic, moreover $b:B\to Y$ can be chosen to be generically finite if $Aut_0(X_y)=\{1\}$. 
\end{remark}

\begin{example}\label{ex var=0} 1. If the smooth fibres $X_y$ of $f$ have $h^{2,0}(X_y)=1$, and if one at least is projective, then $Var(f)=0$ if $a(Y)=0$ (\cite{CS}).

2.If the smooth fibres $X_y$ are hyperk\"ahler, and if one at least is not projective, then $Var(f)=0$ (no condition on $Y$) (\cite{Ca05}). 

3. If $a(f)=q(f)=0$ and if the general fibre of $f$ is Kummer, then $Var(f)=0$ if $a(Y)=0$. More generally, $Var(f)$ factorises through the algebraic reduction of $Y$. (\cite{F83}, Proposition 8.6). 
\end{example}

\subsection{Conditional description of $a(X)=0$.}\label{a=0 I}

We shall prove Theorem \ref{t' a=0} in two steps: first describing $X$ in terms of `simple' manifolds, using conjecture \ref{c simple}.2, recalled below, and then identify the simple manifolds using conjecture \ref{c simple}.1.

\begin{conjecture} Let $f:X\to Y$ be a fibration, $X\in C$. If $X_y$ is simple, and if $a(Y)=0$, $Var(f)=0$. 
\end{conjecture}

\begin{remark}\label{r var} A more general conjecture should be true: if $X_y$ is not uniruled, and if $a(Y)=0$, then $Var(f)=0$.

Notice that if $X_y$ is uniruled, this may fail. For example (\cite{Ca90}, 6.1) if $T$ is a torus with $a(T)=0$, $T^*=Pic^0(T)$ its dual, $L\to T\times T^*=P$ the Poincar\'e line bundle, $p:X:=\Bbb P(L\oplus \mathcal{O}_P)\to P$ the projection, and finally $f=\pi_1\circ p:X=\Bbb P(L\oplus \mathcal{O}_P)\to Y=T$, with $\pi_1:T\times T^*=P\to T$ the projection on the first factor. 
\end{remark}

\begin{lemma}\label{l var ss} If $X_y$ is semi-simple, and if $a(Y)=0$, $Var(f)=0$, assuming Conjecture \ref{c simple}.2. 
\end{lemma}

\begin{proof} Since: $Var(f_1\times_Y \dots _Y\times f_r)=Var(f_1)+\dots+ Var(f_r)$ if $f_i:X\to Y,i=1,\dots,r$, one deduces the claim from Lemma \ref{l Zreg}. 2. \end{proof}

We recall Theorem \ref{t' a=0}:

\begin{theorem}\label{t a=0} Assume Conjecture \ref{c simple}.2. Let $X\in C$ with $a(X)=0$. Let $s:X\to \Sigma$ be the MRC of $X$. Then $s$ has unirational prehomogeneous fibres and $\Sigma$ is isogeneous to a product $S\times K$, where $K$ is Kummer and $S$ semi-simple without Kummer factors. If one also assumes Conjecture \ref{c simple}.1, $S$ is a product of bimeromorphically symplectic irreducible manifolds.
\end{theorem}

The relative version thus holds true (by a standard argument involving the compactness of Barlet-Chow scheme):

\begin{corollary} \label{c a=0} Let $f:X\to Y$ be a fibration in C with $a(X_y)=0$. Let $s_f:X\to \Sigma, s'_f:\Sigma \to Y, f=s'_f\circ s_f$ be the relative MRC of $f$. The smooth fibres of $s_f$ are rational prehomogeneous, and the general fibre $\Sigma_y$ of $s'_f$ is isogeneous to a product  $S_y\times K_y$, where $K_y$ is Kummer and $S_y$ semi-simple without Kummer factor.  
\end{corollary}

\begin{proof} (of Theorem \ref{t a=0}) Let $k':X\to K, k:K\to Y, f=k'\circ k$ be the Kummer reduction of $X$. Let $a':X\to A, a:A\to K$ be the algebraic reduction of $k'=a\circ a'$. After Proposition \ref{p ks}, we have: $q(a)=0$ and $k(a')=a(a')=0$. We can thus construct the semi-simple reduction $s':X\to S,s:S\to A$ of $a'=s\circ s'$. We thus obtain the decomposition $X\to S\to A\to K$ of Corollary \ref{c ks} (with $Y$ a point). 

Let next $b':X\to B, b:B\to S$ be the algebraic reduction of $s':X\to S$. Applying Proposition \ref{p ks} to $X\to B\to S\to A$, we get that $q(B/S)=0$, and $a(X/B)=k(X/B)=0$. We can thus take the semi-simple reduction $\sigma':X\to \Sigma, \sigma:\Sigma \to B$ of $b'=\sigma\circ \sigma'$. Its general fibre $\Sigma_b, b\in B$ is semi-simple without Kummer factor. 

Claim 1: $\Sigma=B$. Apply Conjecture \ref{c var} to $\sigma:\Sigma\to B$. Replacing $B$ by a finite cover $B'$, we get the existence of a semi-simple manifold without Kummer factor $\Phi$ such that $\Sigma\times_B B'=\Phi\times  B'$. Applying Proposition \ref{p ks} to $g:X':=X\times_B B'\to X$, we may suppose that $X=X', B=B', \Sigma=B\times \Phi$. We thus get a map $\Sigma=B\times \Phi\to S\times \Phi$. The general fibres of $S\times \Phi\to A$ are thus semi-simple without Kummer factor (since so are the general fibres of $S/A$), and of dimension $k(X/A)+dim(\Phi)$ contradicting, if $dim(\Phi)>0$,  the fact that $S/A$ is the semi-simple reduction of $X/A)$. Thus $dim(\Phi)=0$ and $\Sigma=B$. 

Claim 2: $X=B$. Indeed, the semi-simple reduction of the general fibre $X_b$  of $X/B$ is $\Sigma_b$.  Since $\Sigma_b$ is a point, $X_b$ is also a point  (after Theorem \ref{red ss}).

We apply now Conjecture \ref{c var} and Proposition \ref{p ks} to $s:S\to A$, and may assume, additionally, that $S=A\times F$ for some semi-simple manifold $F$ without Kummer factor. 

We thus obtain $b:X=B\to S=A\times F$ which has Moishezon fibres with $q=0$, and a fibration $a\times 1_F:A\times F\to K\times F$ with Moishezon fibres with $q=0$. The composite $(a\times 1_F)\circ b:X\to K\times F$ of two fibrations with Moishezon fibres and $q=0$ still has Moishezon fibres with $q=0$, hence unirational prehomogeneous since $a(K\times F)=0$ (\cite{F83}, 2.5, or \cite{Ca85}, \S10, Th\'eor\`eme p. 406). This proves Theorems \ref{t a=0} (and \ref{t' a=0}). \end{proof}





\section{Iterated Moishezon manifolds.}\label{s iM}

We briefly explain the notion of relative algebraic reduction, and the construction of the fibration $f:X\to Y$ by iteration of relative algebraic reductions mentioned in the introduction.

$\bullet$ The relative algebraic reduction. One starts with a fibration $g:X\to Z$ with $X\in C$. The algebraic reduction $a'_g:X\to A_g,a_g:A_g\to Z$ of $g:X\to Z$ is a factorisation $a'_g\circ a_g=g$ which induces on the general fibre $X_z$ of $g$ the algebraic reduction $a'_{g\vert X_z}=a_{X_z}:X_z\to A_{g,X_z}=A_{X_z}$ of $X_z$.

$\bullet$ The relative Albanese map $alb'_g:X\to Alb_g, alb_g:Alb_g\to Z$ of $g=alb_g\circ alb'_g$ is similar, except that the map $alb'_g:X\to Alb_g$ is not necessarily surjective or with connected fibres. However $alb_g:Alb_g\to Z$ is over the smooth fibres $X_y$ of $g$ the actual Albanese torus $Alb_{X_y}$. When $a(X)=a(Z)$, the map $alb'_g:X\to Alg_g$ is surjective, with connected fibres (just as when $Z$ is a point). Moreover, there is no transversal subvariety $V\subsetneq Alb_g$ such that $alb_g(V)=Z$.

\medskip

The relative version of the algebraic reduction is constructed in \cite{Ca80} and in \cite{F81}, the relative Albanese map is constructed in\cite{Ca85} \footnote{In \cite{F81} the relative Albanese map is constructed only when the fibres of $g$ are Moishezon, following Grothendieck, by double dualisation of the relative Picard variety. The general construction in \cite{Ca85} is quite different, Grothendieck approach not being available.}. In both cases, the construction relies in an essential way on the compactness of the components of $BC(X/Z)$.

\medskip

$\bullet$ The fibration $f:X\to Y$ with $a(X_y)=0$ for $y\in Y$ general, and $Y$ without nontrivial subvariety $V\subset Y$ with $a(V)=0$ is constructed as follows by iterating algebraic reductions:

First take $a_X=a_{X,1}:X\to A_X=A_{X,1}$, the algebraic reduction of $X$. Then take the algebraic reduction $a'_{X,2}:X\to A_{X,2},a_{X,2}:A_{2,X}\to A_X$ of $a_X$ if $X\neq A_X$ (i.e: if $X$ is not Moishezon). If $X=A_{X,2}$ that is: if the fibres of $a_X$ are Moishezon, or if the general fibres of $a'_{X,2}$ have algebraic dimension zero\footnote{The algebraic dimension of the fibres of any fibration in C is analytically upper-semicontinuous in the following sense: for any $a\geq 0$, the set of $y\in Y$ such that each component of $X_y$ has algebraic dimension at least $a$ is a countable union of closed analytic subsets of $Y$.}, we are finished: $f=a'_{X,2}, Y=A_{X,2}$. Otherwise, we iterate these steps $r$ times until we get a relative algebraic reduction $f=a_{X,r}:X\to A_{X,r}=Y, a'_{X,r}:A_{X,r}\to A_{X,(r-1)}$ of $a_{X,(r-1)}:X\to A_{X,(r-1)}$ such that the general fibre of $a'_{X,r}$ has algebraic dimension zero (it may be  a point), and $Y=A_{X,r}$ has a fibration $a_{X,2}\circ \dots\circ a_{X,r}:Y\to A_X$, the fibres of each $a_{X,j}$ being Moishezon. We say $Y$ is `iterated Moishezon' of length $r$. The class of iterated varieties  manifolds in C of length $r$ is  denoted $M_r$ (so that $M_1$ is the class of Moishezon manifolds). The class of iterated Moishezon manifolds of length at most $r$ is stable by image, subvarieties, the components of $BC(Y)$ are in this class if $Y$ is (\cite{Ca80}). 

From the preceding construction follows that $X$ is not iterated Moishezon if and only if its general point is contained in a positive-dimensional subvariety of algebraic dimension zero.

$\bullet$ Thus $Y$ is inductively constructed from Moishezon Manifolds. But much more can be said according to \cite{F83} (and partly \cite{Ca85}). Let $a_Y:Y=A_{X,r}\to A_Y$ be the algebraic reduction of $Y$, $a_Y\circ f=a_X:X\to A_Y=A_X$ being thus the one of $X$.

1. The fibres of $a_Y$ are prehomogeneous\footnote{That is: the unit component of their group of automorphisms has a nonempty Zariski open orbit. See \cite{U}, \S.19.6. K\"ahler prehomogeneous manifolds are fibre bundles over their Albanese torus with a unirational prehomogeneous fibre (\cite{L}, theorem 3.22)} if they are Moishezon (\cite{Ca85}, \S 9). Thus $a_Y=\tau\circ u$, where $\tau: T\to A_Y$ is the Albanese reduction of $a_Y$ (\cite{Ca85}, \S 5), and $u:Y\to T$ is- a posteriori-the relative MRC of $a_Y$.

2. In the general case, the fibres of $a_Y$ are towers of prehomogeneous Moishezon manifolds. Inductively permuting their $u's$ and $\tau's$, it is proved\footnote{Independently of \cite{Ca85}.} in \cite{F83} that $a_Y=\tau\circ u$, where $u:Y\to T$ has unirational prehomogeneous fibres, and $\tau:T\to Z$ has smooth fibres which are tori.

\begin{question}\label{q preh} Are the smooth fibres of $g=\tau\circ u$ prehomogeneous?
\end{question}

In short: the single obstruction to algebraicity of length $r$ iterated Moishezon manifolds in C are non-projective tori which are length $(r-1)$ iterated extensions of abelian varieties. This obstruction thus follows from the failure of Poincar\'e reducibility for non projective tori.
\medskip





\section{The general case.}\label{s gc}

 Assuming conjecture \ref{c var} and applying Theorem \ref{t a=0}, we get the following canonical decomposition of any $X$ in C: let $a_X:X\to A_X$ be the algebraic reduction of $X$. Then $a_X=\mu\circ \nu$, where:

$\nu: X\to Y=A_{X,r}, \mu:Y:=A_{X,r}\to A_{X}$ is the iterated Moishezon reduction of $X$, with $a(X_y)=0$ and $Y\in M_r$.

Each of the maps $\mu,\nu$ has a $2$-step factorisation:

$\mu=\tau\circ \rho, \rho:Y\to T, \tau:T\to A_X$ is the Albanese reduction of $\mu$, the smooth fibres of $\tau$ being tori in $M_{(r-1)}$, length $(r-1)$ extensions of Abelian varieties, the fibres of $\rho$ are unirational prehomogeneous manifolds, so that $\rho$ is also the MRC of $a_X$. 

$\nu=\sigma\circ \rho', \rho':X\to \Sigma, \sigma:\Sigma \to Y$ is the relative MRC of $\nu$. The fibres of $\rho'$ are prehomogeneous manifolds. The general fibre $\Sigma_y$ of $\sigma$ is isogeneous to a product $S_y\times K_y$, where $S_y$ is  semi-simple without Kummer factor, and $K_y$ Kummer. If we moreover assume conjecture \ref{c simple}, then $S$ is isogeneous to a product of irreducible bimeromorphically symplectic manifolds.

\begin{remark} This $4$-step decomposition in fact applies in the obvious sense to any fibration $g:X\to Z$ with $X$ in C, just taking the relative versions of $\mu,\nu, \rho', \sigma, \rho, \tau$. 
\end{remark}





\section{Some questions.}\label{q}

Just recall: 

Question \ref{q k}.

Question \ref{q preh}

\begin{question} \label{q bim s} Recall that $X_n\in C$ is said to be irreducible bimeromorphically symplectic (resp. exactly irreducible bimeromorphically symplectic) if $n=2m$ is even, and if  $H^{2,0}(X)$ is generated by some $\sigma$ such that $\sigma^{\wedge m}\neq 0$ (resp. if $X$ is irreducible bimeromorphically symplectic, and if, moreover $h^{p,0}(X)$ is zero if $p$ is odd, and equal to $1$ if $p$ is even. 

Are the two properties equivalent? Under which additional conditions do they become equivalent? \end{question} 

\begin{question} \label{q max} Let $X\in C$ be simple (or simple and irreducible bimeromorphically symplectic, or even simple hyperk\"ahler). Is the number of its positive-dimensional maximal subvarieties finite? 
\end{question}





\end{document}